\documentclass[12 pt]{amsart}
\usepackage{url}
\usepackage{amssymb}
\usepackage{amsmath}
\usepackage{stmaryrd}
\usepackage{siunitx}
\usepackage{commath}
\usepackage{graphicx}
\usepackage[caption=false]{subfig}
\usepackage{enumerate}
\usepackage{enumitem}
\usepackage{todonotes}
\usepackage{pgfplotstable, pgfplots}
\usepackage{booktabs}

\newtheorem{remark}{Remark}
\newtheorem{theorem}{Theorem}

\newtheorem{lemma}{Lemma}
\newcommand{\jump}[1]{\llbracket #1 \rrbracket}
\newcommand{\divergence}{\operatorname{div}}
\renewcommand{\div}{\operatorname{div}}
\newcommand{\curl}{\operatorname{curl}}
\newcommand{\tr}{\operatorname{tr}}
\newcommand{\dev}{\operatorname{dev}}
\newcommand{\id}{\operatorname{Id}}
\newcommand{\Hcd}{H(cd)}
\newcommand{\rr}{\mathbb{R}}
\renewcommand{\d}{\partial}
\newcommand{\DD}{{C^{\infty}_c}}
\newcommand{\mesh}{\mathcal{T}}
\newcommand{\facets}{\mathcal{F}}
\newcommand{\trans}{{\operatorname{T}}}
\newcommand{\hhj}{{\operatorname{hhj}}}

\usepackage{siunitx}
\newcommand{\numeoc}[1]{\num[round-precision=1,round-mode=places, scientific-notation=false]{#1}}

\title[HHJ-like methods for stream function formulations]{A Hellan-Herrmann-Johnson-like method for the Stream function formulation of the Stokes equations in two and three space dimensions}
\author[P.~L.~Lederer]{Philip L. Lederer}
\address{Institute for Analysis and Scientific Computing, TU Wien,
Wiedner Hauptstra{\ss}e 8-10, 1040 Wien, Austria}
\email{philip.lederer@tuwien.ac.at}

\begin{document}
%------------------------------------------------------------------------------
\maketitle
%------------------------------------------------------------------------------
\begin{abstract}
  We introduce a new discretization for the stream function formulation of the incompressible Stokes equations in two and three space dimensions.
  The method is strongly related to the Hellan-Herrmann-Johnson method and is based on the recently discovered mass
  conserving mixed stress formulation [J. Gopalakrishnan, P.L. Lederer, J. Sch\"oberl, IMA Journal of numerical Analysis, 2019] that approximates the velocity in an $H(\divergence)$-conforming space and
  introduces a new stress-like variable for the approximation of the gradient of the velocity within the function space $H(\curl\div)$. The properties of the (discrete) de Rham complex allows to extend this method to a stream function formulation in two and three space dimensions.
  We present a detailed stability analysis in the continuous and the discrete setting where the stream function $\psi$ and its approximation $\psi_h$ are elements of $H(\curl)$ and the $H(\curl)$-conforming N\'ed\'elec finite element space, respectively. We conclude with an error analysis revealing optimal convergence rates for the error of the discrete velocity $u_h = \curl(\psi_h)$ measured in a discrete $H^1$-norm. We present numerical examples to validate our findings and discuss structure-preserving properties such as pressure-robustness. 
\end{abstract}
%------------------------------------------------------------------------------
%------------------------------------------------------------------------------
%\begin{AMS}
%  65N12, 65N15, 65N30, 76D07, 76M10.
%\end{AMS}
%------------------------------------------------------------------------------
\maketitle
\section{Introduction}
\label{sec:introduction}

In this work we present a new discretization of the stream function formulation of the Stokes equations in two and three space dimensions. To this end, let $\Omega \subset \mathbb{R}^d$ with $d =2,3$ be a bounded simply connected domain with a Lipschitz boundary $\partial \Omega$.  The Stokes problem is given by: Find $u:\Omega\rightarrow\mathbb{R}^d$
and $p:\Omega \to \mathbb{R}$ such that
\begin{subequations}
  \begin{alignat}{2}
    -\nu \Delta u + \nabla p &= f && \qquad\text{in } \Omega,\label{eq:stokesone}
    \\
    \divergence(u) &= 0 && \qquad \text{in } \Omega,
    \\
    u &= 0 && \qquad \text{on } \partial\Omega,
  \end{alignat}
  \label{eq:stokes}
\end{subequations}
where $f:\Omega \to \mathbb{R}^d$ is a given body force and $\nu \in \mathbb{R}^+$ is the constant kinematic viscosity. Here, $u$ denotes the velocity of the considered fluid and $p$ is the corresponding (kinematic) pressure. Note, that we only consider homogeneous Dirichlet boundary conditions in this work (see also comment below). Following \cite{Girault:book, BENDALI1985537}, the property $\divergence(u) = 0$ motivates to define the stream function formulation of the Stokes equations given by: Find $\psi: \Omega\rightarrow\mathbb{R}^{d(d-1)/2}$ such that in two dimensions we have
\begin{subequations}
  \begin{alignat}{2}
    -\nu \Delta^2 \psi  &= \curl(f) && \qquad\text{in } \Omega,
    \\
    \psi = \frac{\partial \psi}{\partial n} & = 0 && \qquad \text{on } \partial\Omega,
  \end{alignat}
  \label{eq:streamstokes}
\end{subequations}
and in three dimensions 
\begin{subequations}
  \begin{alignat}{2}
    -\nu \Delta^2 \psi  &= \curl(f) && \qquad\text{in } \Omega,
    \\
    \divergence(\psi) &= 0 && \qquad\text{in } \Omega,
    \\
    \psi \times n  = \curl(\psi) \times n & =0 && \qquad \text{on } \partial\Omega.
  \end{alignat}
  \label{eq:streamstokestwo}
\end{subequations}
Then we have the relation $\curl(\psi) = u$. One of the main attractions of deriving a discrete method for \eqref{eq:streamstokes} and \eqref{eq:streamstokestwo} instead of \eqref{eq:stokes} lies on the hand: Whereas standard mixed finite element methods for the Stokes equations usually enforce the incompressibility constraint only in a weak sense, the discrete velocity solution obtained from a stream function formulation is always exactly divergence-free. This structure preserving property helps for example in the case of convection dominated flows (when we consider the full Navier-Stokes equations, see \cite{Lehrenfeld:2016,Cockburn:2004b,Cockburn:2007b}) and further allows to derive pressure independent velocity error estimates. Such estimates are called pressure robust, see \cite{Linke:2014, Brennecke:2015,Linke:2016c, John:2017,Lederer:2017b}, and are of great interest particularly in the case of vanishing viscosity where an inaccurate pressure approximation might induce a blow up of the velocity. These findings were also extended to the definition of pressure robust error estimators, see \cite{MR3962897} and \cite{MR3240852}. In the latter work the authors presented a residual estimator with the help of the stream function formulation which implicitly lead to pressure robustness. Finally, the stream function formulation recently got popular for the approximation of incompressible fluids on surfaces, see \cite{MR4050536}.

In the derivation of the above equations it was crucial that the domain is simply connected. In a more general setting, the equations and boundary conditions are much more evolved since the potential $\psi$ is not uniquely defined any more, see fore example in \cite{BENDALI1985537} for a detailed discussion. The approximation of the fourth order problem \eqref{eq:streamstokes} (also known as biharmonic problem) requires finite elements of higher regularity.  To overcome this problem, it is common to reformulate the biharmonic problem to the directly related stream function vorticity formulation.
%given by: Find $\psi: \Omega \rightarrow\mathbb{R}, \omega: \Omega \rightarrow\mathbb{R}$ such that
%\begin{alignat*}{2}
%   \Delta \psi  &= \omega && \qquad\text{in } \Omega,\\
%    -\nu \Delta \omega  &= \curl(f) && \qquad\text{in } \Omega.
%  \end{alignat*}
  %In order to close the above equations appropriate boundary conditions have yet to be defined. Following \cite{MR1717814}, this is particularly difficult as whenever the velocity $u$ is fixed on a part of the boundary, this necessarily results in two boundary conditions for $\psi$ and no boundary condition for $\omega$. In \cite{MR553350} the authors, on simply connected domains, were able to deal with this problem by
%  solving an additional problem on the surface of the boundary. Still,
Many authors have studied this problem, see for example \cite{MR1717814, MR553350, MR1726723, MR1186736,nedelecstream}. We also want to cite the very recent work \cite{MR4068857} an the references therein for a further discussion on the connection of the stream function and stream function vorticity formulation and the occurring boundary conditions. Regarding our choice of homogeneous Dirichlet boundary conditions in this work we want to mention, that this is a non trivial case as it was discussed in detail in \cite{veclapDBC}. Therein the authors show, that one might loose optimal convergence when mixed finite element methods including the vorticity are used. However, we want to emphasize that the methods proposed in this work are of optimal order. In \cite{BENYU19971,2587098,MR1825701}, the authors focused on the pure stream function formulation (in two space dimensions) given by \eqref{eq:streamstokes} and derived a finite difference scheme for the approximation of the bi-Laplacian operator. A mixed finite element method was derived in \cite{CIARLET1974125}. Due to the huge success of discontinuous Galerkin methods (DG) for elliptic problems, the techniques were also applied to fourth order problems, see \cite{MR2298696, MR1915664, MR2048235, MR2329348,MR3240852}. 
  
We particularly want to mention the works \cite{hellan,herrmann,johnson, comodi}, since they play a key role in the derivation of the methods introduced in this work. Therein the authors derived a mixed method, also known as the Hellan-Herrmann-Johnson method (HHJ), by introducing an auxiliary variable to approximate the matrix valued symmetric gradient. This has many advantages as it results for example in a reduced coupling in the finite element system matrix compared to a DG formulation and that no second order differential operators have to be explicitly implemented in the finite element code. In \cite{Girault:book}, the authors showed that this techniques can also be used to approximate the stream function formulation and presented a detailed analysis. Nevertheless, the authors claim that the extension to the three dimensional case is not straight forward.

This work is dedicated to fill this gap. To this end we first introduce a modified (rotated) version of the HHJ-method in two dimensions which can then be easily extended to the three dimensional setting. This is possible since the new modified HHJ-method can be interpreted as a discrete stream function formulation of the mass conserving mixed stress formulation (MCS) defined in \cite{Lederer:2019b,Lederer:2019c,lederer2019mass}. The MCS method approximates the discrete velocity $u_h$ in an $H(\divergence)$-conforming finite element space and the discrete pressure in the appropriate ($L^2$-conforming) space of piece wise polynomials. This leads to exactly divergence-free velocity approximations, i.e. $\div(u_h)=0$. The properties of the discrete de Rham complex then motivates to define a discrete stream function $\psi_h$ in an $H(\curl)$-conforming finite element space such that $\curl(\psi_h) = u_h$, which leads to the resulting modified HHJ-method.
%This has another positive effect: whereas the  two dimensional system \eqref{eq:streamstokes} is no saddle point problem anymore (compared to \eqref{eq:stokes}), this is not the case in three space dimensions due to the additional divergence constraint in \eqref{eq:streamstokestwo}, which makes it less attractive for an approximation. Nevertheless, since the modified HHJ-method approximates $\psi$ in an $H(\curl)$-conforming finite element space, this allows to include the constraint directly in the finite element space, thus the advantage of the two dimensional case also holds in three dimensions.
  
  %Note, that although the modified HHJ-method is a mixed method, a standard hybridization technique allows to apply a static condensation process such that the resulting system is again symmetric positive definite. Whereas this seems to be obvious for the two dimensional case, the additional divergence constraint in \eqref{eq:streamstokestwo}, make things more complicated in three dimensions. Nevertheless, the approximation of $\psi$ in an $H(\curl)$-conforming finite element space allows to include this constraint directly in the finite element basis and thus, again a symmetric positive system can be derived.

  Finally note, that there exists a similar connection between the (standard) HHJ-method and the tangential-displacement and normal-normal-stress continuous mixed finite element method for elasticity \cite{MR3712290,MR2826472,Pechstein2018} which motivated the definition of the methods within this work.

The paper is organized as follows. In Section 2 we define the basic notation and symbols that we shall use throughout this work. In Section 3 we discuss the classical weak formulation and a new weak formulation with reduced regularity of the stream function formulation, and present a detailed stability analysis. Section 4 is dedicated to the derivation of the new modified HHJ-method in two and three space dimensions. The technical details needed to prove discrete stability and convergence of the error in appropriate norms are included in Section 5. In Section 6 we present a simple post processing for a pressure discretization. We conclude the work with Section 7 where we present numerical examples to illustrate the theory.
%\begin{enumerate}
%\item motivation: new streaum function formulation in the cont setting allows -> do not need fem for forth order problems
%\item Static condensation -> gives spd system??
%\item gives an answer to the 3d case in Girault Raviart
%\item Exact divergence free solutions -> pressure robustness (pressure independent error estimates
%\item Preconditioner on exactly divergence free subspace
%\item Error estimator: Kantschat hat gezeigt dass das mit stream function funktioniert. Jetzt will man das lokal equilibrieren -> dadurch umgeht man das loecherproblem
%\item symmetry with GG bubbles?
%\end{enumerate}
  \section{Preliminaries}
 Let $\DD(\Omega)$ denote the set of infinitely differentiable compactly supported real-valued functions on $\Omega$ and let $(\DD)'(\Omega)$ denote the space of distributions as usual. In this work we include the range in the notation, hence \
\begin{align*}
  \DD(\Omega, \rr^d) = \{ u: \Omega \to \rr^d : u_i \in \DD(\Omega)\} \\
  \DD(\Omega, \rr^{d \times d}) = \{ u: \Omega \to \rr^{d \times d}: u_{ij} \in \DD(\Omega)\},
\end{align*}
represent the vector valued and matrix-valued versions of $\DD(\Omega) = \DD(\Omega, \rr)$. This notation is  extended in an obvious fashion to other function spaces as needed.

Depending on the type of the function, the gradient~$\nabla$ is to be understood from the context as an operator that results in either a vector whose components are $[\nabla \phi]_i = \d_i \phi$ (where $\d_i$ is the partial derivative $\d/\d x_i$) for $\phi \in \DD'(\Omega, \rr)$ or a
matrix whose entries are $[\nabla \phi]_{ij} = \d_j \phi_i$ for
$\phi \in \DD'(\Omega, \rr^d)$. Similarly, the ``curl'' is given as any of the following three differential operators
\begin{align*}
  \curl(\phi)
  & = (-\partial_2 \phi, \partial_1 \phi)^\trans, 
  && \text{ for } \phi \in \DD'(\Omega, \rr) \text{ and } d=2,
  \\
  \curl( \phi)
  & = -\partial_2 \phi_1+ \partial_1 \phi_2,
  && \text{ for } \phi \in \DD'(\Omega, \rr^2) \text{ and } d=2,
  \\
  \curl (\phi)
  & 
    = \nabla \times \phi   
  && \text{ for } \phi \in \DD'(\Omega, \rr^3) \text{ and } d=3,     
\end{align*}
where $(\cdot)^\trans$ denotes the transpose.  Finally, we define the $\divergence(\phi)$ as either $\sum_{i=1}^d \d_i \phi_i$ for
vector-valued $\phi \in \DD'(\Omega, \rr^d),$ or the row-wise divergence
$\sum_{j=1}^d \d_j \phi_{ij}$ for matrix-valued $\phi \in \DD'(\Omega, \rr^{d \times d})$.

We denote by $L^2(\Omega, \rr)$ the space of square-integrable functions on $\Omega$ and by $H^m(\Omega, \rr)$ the standard Sobolev space of order $m$. In particular we further use the space $H_0^1(\Omega, \rr^d) = \{ v \in H^1(\Omega, \rr^d): v|_{\partial \Omega} = 0 \}$. With $\tilde t = d(d-1)/2$ and the above (weak) differential operators we then further define the Sobolev spaces 
\begin{align*}
  H(\divergence,\Omega)
  &= \{ v \in L^2(\Omega, \rr^d): \divergence(v) \in
    L^2(\Omega) \},
  \\
  H(\curl,\Omega) 
  &= \{ v \in L^2(\Omega, \rr^d):
    \curl(v) \in L^2(\Omega,
    \rr^{\tilde t}) \}.
\end{align*}
Similarly as before, we denote by $H_0(\divergence,\Omega) = \{ v \in  H(\divergence,\Omega): v\cdot n|_{\partial \Omega} = 0\}$ and $H_0(\curl,\Omega) = \{ v \in  H(\divergence,\Omega): v \times n|_{\partial \Omega} = 0\}$, where $n$ denotes the outward unit normal vector on $\partial \Omega$. On the above spaces we use the standard symbols for the norms, but we will omit the domain $\Omega$ to simplify the notation, i.e. while $\| \cdot \|_{L^2}$ is the $L^2$-norm on $\Omega$ we denote by $\| \cdot\|_{H^1}, \| \cdot\|_{H(\divergence)}$ and $\| \cdot\|_{H(\curl)}$ the norms on $\Omega$ of the spaces $H^1, H(\divergence)$ and $H(\curl)$, respectively. Other non standard spaces and definitions are defined later in the work when they appear in a proper context.

Finally, in this work we use $A \sim B $ to indicate that there are 
constants $c, C>0$ that are independent of the mesh size $h$ (defined in Section \ref{sec::discretisation}) and the viscosity $\nu$ such that $c A \le B \le C A$.  We also use $A \lesssim B$ when there is a $C>0$ independent of $h$ and $\nu$ such that $A \le C B$. In the same manner we also define the symbol $\gtrsim$.

%A well-known trace theorem permits us to define
%$H_0(\div,\Omega) = \{ u \in H(\div, \Omega): u \cdot n|_\Gamma =0\}.$ Here,
%$n$ denotes the outward unit normal on $\Gamma$. In other occurrences,
%it may denote the unit outward normal on boundaries of other domains
%determined from context.

\section{The continuous setting}
\subsection{Weak formulations of the Stokes equations}
 Defining the spaces
\begin{align*}
  X := H_0^1(\Omega, \rr^d), \quad \textrm{and} \quad Q := \{ q \in L^2(\Omega, \rr): \int_\Omega q \dif x = 0 \},
\end{align*}
and assuming regularity $f \in L^2(\Omega, \rr)$, the weak formulation of \eqref{eq:stokes} is given by: Find $(u,p) \in X \times Q$ such that
  \begin{subequations}
  \begin{alignat}{2}
    \int_\Omega \nu \nabla u : \nabla v \dif x - \int_\Omega \divergence(v)  p \dif x &= \int_\Omega f \cdot v && \qquad\forall v \in X
    \\
    \int_\Omega \divergence(u)q  \dif x&= 0 && \qquad \forall q \in Q.
  \end{alignat}
  \label{eq:weakstokes}
  \end{subequations}
  In the recent work \cite{lederer2019mass}, a new weak formulation of the Stokes equations was derived that used a weaker regularity assumption of the velocity space. The idea is motivated by introducing te auxiliary variable $\sigma:= \nu\nabla u$. Defining the trace of a matrix $\tr{\tau} :=\sum_{i+1}^d \tau_{ii}$ and the deviator $\dev{\tau} := \tau - (\tr{\tau}/d) \id $ we have 
  \begin{align*}
    \dev{\sigma} = \dev{\nu \nabla u} = \nu \nabla u - \frac{\nu}{d} \tr{\nabla u} \id = \nu(\nabla u - \frac{1}{d} \divergence (u) \id) = \nu \nabla u,
  \end{align*}
  since $\divergence(u)=0$. By that we can reformulate equations \eqref{eq:stokes} as
  \begin{subequations}
  \begin{alignat}{3}
    \nu^{-1} \dev{\sigma} - \nabla u & = 0  \quad &&\textrm{in } \Omega ,   \label{eq:mixedstokesone}\\
    \divergence (\sigma) - \nabla p & = -f \quad &&\textrm{in }  \Omega,  \label{eq:mixedstokestwo}\\
    \divergence(u) & = 0\quad &&\textrm{in }  \Omega,  \label{eq:mixedstokesthree}\\
    u & = 0 \quad &&\textrm{on } \partial \Omega.  \label{eq:mixedstokesfour}
  \end{alignat}
\label{eq:mixedstokes}
  \end{subequations}
  To derive a variational formulation of equation \eqref{eq:mixedstokes}, we define the velocity space $V := H_0(\divergence, \Omega)$ and the matrix valued function spaces
  \begin{align*}  
    H(\curl \divergence) &:= \{ \tau \in L^2(\Omega, \rr^{d \times d}): \divergence(\tau) \in V'\}, \\
    \Sigma &:= \{ \tau \in H(\curl \divergence): \tr{\tau}=0 \}.
  \end{align*}
  Note, that a proper norm on $H(\curl \divergence)$ is defined by $\| \tau \|^2_{\Hcd} := \| \tau \|_{L^2}^2 + \| \divergence(\tau) \|^2_{V'}$, where $V'$ denotes the dual space and $\| \cdot \|_{V'}$ the corresponding dual space norm. The definition of the space $\Sigma$ is motivated by the distributional divergence of an arbitrary function $\tau \in \Sigma$ given by
 \begin{align*}
    \langle{ \divergence(\tau),  \varphi}\rangle_{V} = - \int_\Omega \tau : \nabla \varphi \dif x \quad \forall \varphi \in  \DD(\Omega, \rr^d).
 \end{align*}
 Hence, by a density argument of smooth functions in $H_0(\div, \Omega)$, a weak formulation of \eqref{eq:mixedstokesone} is given by $\nu^{-1}\int_\Omega \sigma:\tau + \langle{ \divergence(\tau),  u}\rangle_{V} = 0$, where $\langle \cdot , \cdot \rangle_V$ denotes the duality bracket on $V' \times V$. For more details we refer to chapter 4 in \cite{lederer2019mass}. Using similar arguments for the other lines of \eqref{eq:mixedstokes}, the mass conserving mixed stress formulation (MCS) then reads as: Find $(\sigma, u ,p) \in V \times \Sigma \times Q$ such that
    \begin{subequations}
  \begin{alignat}{2}
    \int_\Omega \frac{1}{\nu} \sigma: \tau \dif x  + \langle{ \divergence(\tau),  u}\rangle_{V}  & = 0&&    \quad \forall \tau \in \Sigma, \label{eq::mixedstressstokesweakone}  \\
    \langle \divergence(\sigma), v\rangle_{V}  + \int_\Omega\divergence(v) p \dif x  & = -\int_\Omega  f \cdot v \dif x    &&    \quad \forall v \in V,  \label{eq::mixedstressstokesweaktwo} \\ 
    \int_\Omega \divergence(u) q \dif x &=0  &&    \quad \forall q \in Q.\label{eq::mixedstressstokesweakthree} 
\end{alignat}
\label{eq::mixedstressstokesweak}                                    
  \end{subequations}
  where we used that $\dev = \id$ for functions in $\Sigma$. Uniqueness and existence of \eqref{eq::mixedstressstokesweak} in $V \times \Sigma \times Q$ (with the corresponding natural norms) was proven in Section 4.3.1 in~\cite{lederer2019mass}. Note, that the velocity solution $u \in V$ of equation \eqref{eq::mixedstressstokesweak} has a reduced regularity in contrast to the velocity solution of the standard weak formulation of the Stokes equation given by \eqref{eq:weakstokes}.
  \begin{remark} \label{rem::bndcond}
    The homogeneous Dirichlet boundary conditions \eqref{eq:mixedstokesfour} were implicitly split into a normal and a tangential part. Whereas the homogeneous normal Dirichlet values are incorporated as essential boundary conditions in the space $V = H_0(\divergence,\Omega)$, the homogeneous tangential Dirichlet values are included as natural boundary conditions in \eqref{eq::mixedstressstokesweakone}, see also Section 4.3 in \cite{lederer2019mass}.
  \end{remark}
  
  \subsection{The stream function formulation}
  \subsubsection{The standard weak formulation} \label{sec::stdstreamform}
  In this section we summarize the findings of Section 5.2 in \cite{Girault:book} to derive the standard variational stream function formulations of the Stokes problem. 
  By the incompressibility constraint $\divergence(u) = 0$, the velocity solution $u$ of the Stokes equation \eqref{eq:stokes} can be expressed as the $\curl$ of a scalar-valued ($d = 2$) or vector-valued ($d = 3$) potential $\psi$ called the stream function, i.e. $\curl(\psi) = u$. To this end we define for $d = 2$ 
  \begin{align*}
    \Psi &:= \{\phi \in H^2(\Omega, \rr): \phi|_{\partial \Omega} = 0, \frac{\partial \phi}{\partial n}\Big|_{\partial \Omega} = 0 \},
  \end{align*}
  and for $d=3$
  \begin{align*}  
        \Psi &:= \{\phi \in L^2(\Omega, \rr^3): \div(\phi) \in H^1(\Omega, \rr), \curl(\phi) \in H_0^1(\Omega, \rr^3), \phi \times n |_{\partial \Omega} = 0 \}.
  \end{align*}
The stream function can be characterized as the unique solution of the weak problem: Find $\psi \in \Psi$ such that
  \begin{alignat}{2}  \label{eq:stream}
    \int_\Omega \Delta \psi : \Delta \phi \dif x &= \int_\Omega f \cdot \curl(\phi) && \quad \forall \phi \in \Psi.
  \end{alignat}
  According to Theorem 5.5 and Lemma 5.1 in \cite{Girault:book}, equation \eqref{eq:stream} has an unique solution such that $\curl(\psi) = u $ is the unique solution of the standard variational formulation of the Stokes equation \eqref{eq:weakstokes}. %Further, in three dimensions we have that $\divergence(\psi)=0$.
  In contrast to \eqref{eq:streamstokestwo}, the weak formulation \eqref{eq:stream} does not directly include the constraint $\div(\psi) = 0$ in three space dimensions, as this follows implicitly due to the choice of $\Psi$, see Lemma 5.1 in \cite{Girault:book}.
  %In \cite{BENDALI1985537} the authors showed that \eqref{eq:stream} can be interpreted as a weak formulation of \eqref{eq:streamstokestwo} in $H^{-2}( \Omega, \rr^3)$.
  \subsubsection{A weak formulation with reduced regularity} \label{sec::redregstreamform}
  In the following we derive a new weak formulation for the stream function formulation that is motivated by the MCS formulation given by equation \eqref{eq::mixedstressstokesweak}.

  We start with the case $d = 3$. As discussed in the previous section, the solution of the MCS formulation given by equation \eqref{eq::mixedstressstokesweak}, fulfills the weaker regularity $u \in H_0(\divergence, \Omega)$. From the properties of the de Rham complex, see for example in \cite{Boffi:book}, the divergence constraint $\div(u) = 0$ then motivates the existence of a vector potential $\psi \in H_0(\curl, \Omega)$ such that $\curl(\psi) = u$. Similarly as before, we then further introduce a new variable $\sigma \in \Sigma$ such that $\sigma = "\nabla \curl(\psi)"$ (see equation \eqref{eq:mixedstokesone}) in a weak sense since the gradient is not well-defined for $\curl(\psi)$. For an arbitrary $\tau \in \Sigma$ this leads to 
\begin{align*}
  \int_\Omega \frac{1}{\nu} \sigma : \tau \dif x + \langle \div(\tau), \curl(\psi) \rangle_V = 0.
\end{align*}
Note, that the duality pair is well defined for the function $\curl(\psi)$ since 
\begin{align} \label{eq::admiss}
  \div(\curl(\phi)) = 0, \quad \textrm {and} \quad
  \curl(\phi) \cdot n = \div_{\partial \Omega}(\phi \times n) = 0 \quad \forall \phi \in H_0(\curl),
\end{align}
where $\div_{\partial \Omega}$ is the surface divergence, see Section 2.1 in \cite{Boffi:book}. This shows that $\curl(\phi) \in H_0(\div, \Omega) = V$ for all $\phi \in H_0(\curl)$, hence $\langle \div(\tau), \curl(\psi) \rangle_V$ is well defined. Similarly, we derive a weak formulation of the momentum equation \eqref{eq:mixedstokestwo} by testing with a function $\curl(\phi)$ with $\phi \in H_0(\curl)$ to get
\begin{align*}
  \langle \div(\sigma), \curl(\phi) \rangle_V = - \int_\Omega f \cdot \curl(\phi) \dif x,
\end{align*}
where we used integration by parts and that $\curl(\nabla p)) = 0$, hence the pressure integral disappeared. Finally, to uniquely determine the vector potential $\psi$, we introduce a gauging as it is also known for mixed formulation of the Maxwell's equations, see for example in \cite{monk2003finite}. To this end we demand that $\psi$ is orthogonal on gradient fields which mimics the conditions $\div(\psi) = 0$ in \eqref{eq:streamstokestwo}. Introducing the spaces 
\begin{align}
  W := H_0(\curl, \Omega) \quad \textrm{and } \quad S := H^1_0(\Omega, \rr),
\end{align}
we then have the weak formulation: Find $(\psi, \sigma, \lambda) \in W \times \Sigma \times S$ such that
\begin{subequations}
  \begin{alignat}{2}
    \int_\Omega \frac{1}{\nu} \sigma: \tau \dif x  + \langle{ \divergence(\tau),  \curl(\psi)}\rangle_{V}  & = 0&&    \quad \forall \tau \in \Sigma, \label{eq::mixedstressstreamweakone}  \\
    \langle \divergence(\sigma), \curl(\phi) \rangle_{V}  + \int_\Omega \phi \cdot \nabla \lambda & = -\int_\Omega  f \cdot \curl(\phi) \dif x    &&    \quad \forall \phi \in W,  \label{eq::mixedstressstreamweaktwo} \\ 
    \int_\Omega \psi \cdot \nabla \mu \dif x &=0  &&   \quad \forall \mu \in S.\label{eq::mixedstressstreamweakthree} 
\end{alignat}
\label{eq::mixedstressstreamweak}                                    
\end{subequations}
\begin{remark}
Similarly as for the solution of \eqref{eq::mixedstressstokesweak}, the homogeneous Dirichlet boundary conditions in \eqref{eq::mixedstressstreamweak} are split into a normal and a tangential part, see Remark~\ref{rem::bndcond}. The homogeneous normal condition $u \cdot n = \curl(\phi) \cdot n = 0$ is incorporated as an essential boundary condition in the space $W$, see equation \eqref{eq::admiss}, whereas the boundary condition $ u \times n = \curl(\phi) \times n = 0$ is induced as a natural boundary condition in \eqref{eq::mixedstressstreamweakone}.
\end{remark}
\begin{theorem}\label{th::exstream} Let $d = 3$. There exists an unique solution $(\psi, \sigma, \lambda) \in W \times \Sigma \times S$ of the weak formulation \eqref{eq::mixedstressstreamweak} such that
  \begin{align*}
    \| \sigma \|_{\Hcd} + \| \psi \|_{H(\curl)} + \| \nabla \lambda \|_{L^2} \lesssim \| f \|_{L^2}.
  \end{align*}
  Further, the velocity solution $u := \curl(\psi) \in V $ is the solution of \eqref{eq::mixedstressstokesweak}.
\end{theorem}
\begin{proof}
  The proof is based on the standard theory of mixed problems, see \cite{Boffi:book}. Continuity of the bilinear forms follows immediately with the Cauchy Schwarz inequality, the continuity of the duality bracket and that for all $\phi \in W$ there holds the estimate
  \begin{align*}
    \| \curl(\phi) \|^2_V = \|\curl(\phi)\|^2_{L^2} + \|\divergence (\curl(\phi))\|^2_{L^2} = \|\curl(\phi)\|^2_{L^2} \le \| \phi \|^2_{H(\curl)}.
  \end{align*}
  We continue with the kernel ellipticity. To this end let $(\sigma,\lambda) \in \Sigma \times S$ such that
  \begin{align} \label{eq::kernel}
    \langle \divergence(\sigma), \curl(\phi) \rangle_{V}  + \int_\Omega \phi \cdot \nabla \lambda  =0 \quad \forall \phi \in W.
  \end{align}
  In a first step we bound the norm of $\lambda$: Since $\nabla S \in W$, equation~\eqref{eq::kernel} shows
  %For this note, that for an arbitrary $\mu \in S$ there holds by the properties of the de Rham complex that $\nabla \mu \in W$, and thus using equation~\eqref{eq::kernel} we get
  \begin{align*}
    %\| \nabla \lambda \|_{L^2} = \sup\limits_{\mu \in S}\frac{\int_\Omega \nabla \mu \cdot \nabla \lambda \dif x}{\| \nabla \mu \|_{L^2}} = \sup\limits_{\mu \in S}\frac{- \langle \divergence(\sigma), \curl(\nabla \mu) \rangle_{V}}{\| \nabla \mu \|_{L^2}} = 0,
    \| \nabla \lambda \|^2_{L^2} =  \int_\Omega \nabla \lambda \cdot \nabla \lambda \dif x = - \langle \divergence(\sigma), \curl(\nabla \lambda) \rangle_{V}= 0,
  \end{align*}
 thus $\lambda = 0$ (due to the zero boundary conditions of $\lambda$). Next, let $u \in V$ be arbitrary. Then, using a regular decomposition of $u$ there exist functions $\theta,z \in H_0^1(\, \rr^d)$ such that $u = \curl(\theta) + z$ and
  \begin{align*}
    \| \theta\|_{H^1}  +  \| z\|_{H^1} \lesssim \| u \|_{V}.  
  \end{align*}
  Many authors have stated such decomposition results under various assumptions on the domain $\Omega$. Under the current assumptions we refer for example to \cite{DemloHiran14}. With the decomposition result we then get
  \begin{align*}
    \| \div(\sigma)\|_{V'} &= \sup\limits_{u \in V}\frac{\langle \div(\sigma), u \rangle_V}{\|u \|_V} \sim \sup\limits_{\theta,z \in H_0^1(\Omega, \rr^d)}\frac{\langle \div(\sigma), \curl(\theta) + z \rangle_V}{\|z \|_{H^1} + \|\theta \|_{H^1}}.
  \end{align*}
  Since $H_0^1(\Omega, \rr^d) \subset W$, equation~\eqref{eq::kernel} and $\lambda = 0$ implies that $\langle \div(\sigma), \curl(\theta) \rangle_V = 0$. By the definition of the distributional divergence and the Cauchy Schwarz inequality this then finally leads to
    \begin{align*}
              \| \div(\sigma)\|_{V'} &\sim  \sup\limits_{z \in H_0^1(\Omega , \rr^d)}\frac{\langle \div(\sigma), z \rangle_V}{\|z \|_{H^1}}
                           = \sup\limits_{z \in H_0^1(\Omega, \rr^d)}\frac{-\int_\Omega \sigma : \nabla z \dif x }{\|z \|_{H^1}}\le \| \sigma\|_{L^2},
    \end{align*}
    thus in total $\|\sigma \|^2_{\Hcd} + \|\nabla \lambda \|_{L^2}^2 \lesssim \int_\Omega \sigma : \sigma \dif x$, which proves kernel ellipticity.

    It remains to prove the inf-sup condition. To this end let $\phi \in W$ be arbitrary. Then, as before, we use a regular decomposition (now in $H_0(\curl)$) of $\phi$ to find functions $\theta \in H_0^1(\Omega, \rr)$ and $z \in H_0^1(\Omega, \rr^d)$ such that $\phi = \nabla \theta + z$ and
  \begin{align} \label{eq::regdccurl}
    \| \nabla z\|_{L^2}  \le c_1 \| \curl(\phi) \|_{L^2} \quad \textrm{and} \quad \| \nabla \theta \|_{L^2}  \le c_2  \| \phi \|_{W},
  \end{align}
  where $c_1>0$ and $c_2>0$ are two constants. With the triangle inequality we then also have $\| \nabla \theta\|_{L^2}^2 +  \| \nabla z \|_{L^2}^2 \sim  \| \phi \|^2_{W}$. Next, we use Lemma 12 of \cite{lederer2019mass} which states that for every $ u \in V$ there exists a $\sigma \in \Sigma$ such that $\langle \divergence(\sigma), u \rangle_V \ge c_3 \| u \|_V^2$, with a constant $c_3 > 0$, and $\| \sigma \|_{\Hcd} \lesssim \| u \|_V$. Since $\curl(\phi) \in V$, this gives for $\sigma_1 = \frac{c_1}{c_3} \sigma$ and $\lambda = \theta$, the estimate
  \begin{align*}
    \langle \divergence(\sigma_1), \curl(\phi) \rangle_V + \int_\Omega \phi \cdot \nabla \lambda \dif x
    &= \frac{c_3}{c_1} \langle \divergence(\sigma), \curl(\phi) \rangle_V + \int_\Omega \phi \cdot \nabla \theta \dif x \\
    &\ge c_1 \| \curl(\phi) \|^2_{L^2} + \int_\Omega \nabla \theta \cdot \nabla \theta \dif x + \int_\Omega z \cdot \nabla \theta \dif x \\
    &\ge  \| \nabla z \|^2_{L^2} + \| \nabla \theta \|_{L^2}^2 -  \| \nabla z \|_{L^2}\| \nabla \theta \|_{L^2}   \gtrsim \|\phi\|^2_W,
  \end{align*}
  where we used the Cauchy Schwarz inequality, the left equation of \eqref{eq::regdccurl} and the Young inequality for the last term. The continuity estimates
 % \begin{align*}
 %   \langle \divergence(\sigma_1), \curl(\phi) \rangle_V + \int_\Omega \phi \cdot \nabla \lambda \dif x \gtrsim \|\phi \|_W^2.
%  \end{align*}
$\| \sigma_1 \|_{\Hcd} \lesssim \|\curl(\phi)\|_{L^2}$ and the right equation of \eqref{eq::regdccurl} (for $\lambda = \theta$) then shows
  \begin{align*}
   \langle \divergence(\sigma_1), \curl(\phi) \rangle_V + \int_\Omega \phi \cdot \nabla \lambda \dif x\gtrsim \|\phi\|_W (\| \sigma_1 \|_{\Hcd} + \| \nabla \lambda \|_{L^2}),
  \end{align*}
thus the inf-sup condition is proven and we conclude the proof of the existence and continuity result.

    Now let $(\sigma^{\operatorname{MCS}}, u^{\operatorname{MCS}}, p^{\operatorname{MCS}}) \in \Sigma \times V \times Q$ be the solution of equation \eqref{eq::mixedstressstokesweak}, and set $u := \curl(\psi) \in V$. As \eqref{eq::mixedstressstokesweakthree} gives $\div(u^{\operatorname{MCS}})=0$, it follows that the pair $(\sigma^{\operatorname{MCS}}, u^{\operatorname{MCS}})$ is uniquely defined by testing equation \eqref{eq::mixedstressstokesweaktwo} only with divergence free test functions $v \in V^0:=\{v \in V: \divergence(v) = 0\}$. Since the pair $(\sigma^{\operatorname{MCS}}, u)$ is also a solution of \eqref{eq::mixedstressstokesweaktwo} (on $V^0$), we have that $u = u^{\operatorname{MCS}}$ by the uniqueness of the solution of equation \eqref{eq::mixedstressstokesweak}.
  \end{proof}
  
  For the case $d=2$, a similar formulation can be derived. Now, the de Rham complex motivates the existence of a scalar potential $\psi \in H^1_0(\Omega, \rr)$ such that $\curl(\psi) = u$. Note, that $\psi$ is already uniquely defined thus no further gauging is needed. With the same observations as for the three dimensional case, we then have the weak formulation: Find $(\psi, \sigma) \in S \times \Sigma$ such that
\begin{subequations}
  \begin{alignat}{2}
    \int_\Omega \frac{1}{\nu} \sigma: \tau \dif x  + \langle{ \divergence(\tau),  \curl(\psi)}\rangle_{V}  & = 0&&    \quad \forall \tau \in \Sigma, \label{eq::mixedstressstreamweakonetwo}  \\
    \langle \divergence(\sigma), \curl(\phi) \rangle_{V} & = -\int_\Omega  f \cdot \curl(\phi) \dif x    &&    \quad \forall \phi \in S.  \label{eq::mixedstressstreamweaktwotwo}
\end{alignat}
\label{eq::mixedstressstreamweaktwod} 
\end{subequations}
\begin{theorem} Let $d = 2$. There exists a unique solution $(\psi, \sigma) \in S \times \Sigma$ of the weak formulation \eqref{eq::mixedstressstreamweaktwod} such that
  \begin{align*}
    \| \sigma \|_{\Hcd} + \| \psi \|_{H(\curl)} \lesssim \| f \|_{L^2}.
  \end{align*}
    Further, the velocity solution $u := \curl(\psi) \in V $ is the solution of \eqref{eq::mixedstressstokesweak}.
\end{theorem}
\begin{proof}
  The proof follows with the same steps as in the proof of Theorem \ref{th::exstream}.
\end{proof}
\section{A Hellan-Herrmann-Johnson like method for the stream function formulation} \label{sec::discretisation}
In this section we present a new discretization of the weak formulation for the stream function with reduced regularity, see equation \eqref{eq::mixedstressstreamweak} and \eqref{eq::mixedstressstreamweaktwod}. In two space dimensions the method can be interpreted as a rotated version of the Hellan-Herrmann-Johnson (HHJ) formulation for fourth order problems, see \cite{hellan,herrmann,johnson, comodi,babuska}.

We start with some preliminaries for the discrete setting. Let $\mesh$ be a shape regular quasi uniform triangulation (partition) of the domain $\Omega$, which consists of triangles and tetrahedrons in two and three dimensions, respectively. We denote by $h$ the  maximum of the diameters of all elements in $\mesh$. Since $\mesh$ is quasi uniform we have $h \approx \textrm{diam}(T)$ for any $ T \in \mesh$. 
The set of element interfaces and boundaries is given by $\facets$.
% and is split into facets on the domain boundary $F \subset \facets \cap \partial \Omega =:\facets^{\operatorname{ext}}$ and facets in the interior $F \subset \facets \cap \Omega =:\facets^{\circ}$.
On each facet $F \in \facets$ we denote by $\jump{\cdot}$ the usual jump operator. For facets on the boundary the jump operator is just the identity.

For readability, we use again the symbol $n$ for the outward unit normal vector on the element boundaries and on facets on the global boundary. Then, the normal and tangential trace of a smooth enough function $\phi: \Omega \rightarrow \rr^d$ is defined by 
\begin{align*}
 \phi_n := \phi \cdot n \quad \textrm{and} \quad \phi_t := \phi - \phi_n n.
\end{align*}
Note, that this definition gives a scalar normal trace and vector valued tangential trace. In two dimensions we further define by $t$ the unit tangent vector that is obtained by rotating $n$ anti-clockwise by 90 degrees (thus $t = n^\perp$), so that $\phi_t = (\phi \cdot t) t$.  In a similar manner for a smooth enough function $\phi: \Omega \to
 \rr^{d \times d}$ we set
\begin{align*}
  \phi_{nn} =   \phi : (n \otimes n) = n^\trans \phi n \quad \textrm{and} \quad \phi_{nt} =  \phi n -  \phi_{nn} n,
\end{align*}
which reads as a scalar ``normal-normal component'' and a vector-valued ``normal-tangential component''. In two dimensions we may write $\phi_{nt} = (t^\trans\phi n)t$.

Next, let $T \in \mesh$ be arbitrary, then we define by $\mathcal{P}^k(T, \rr)$ scalar-valued polynomials of total order $k$ on $T$. On the triangulation we then correspondingly set $\mathcal{P}^k(\mesh, \rr) := \cup_{T \in \mesh}\mathcal{P}^k(T, \rr)$. These definitions are extended to vector and tensor valued polynomials as before. 

Finally, let $\omega \subset \Omega$ be an arbitrary subset, then we use $(\cdot, \cdot)_\omega$ for the $L^2$-inner product on $\omega $ and by $|| \cdot ||^2_\omega := ( \cdot,\cdot)_\omega$ the corresponding norm.
\subsection{The HHJ-method in two dimensions}~\\
In this section we summarize the derivation of the HHJ-method for the stream function formulation as it is described in Section 4.1 in \cite{Girault:book}. In this section we only consider the case $d=2$.

The idea of the HHJ formulation is, similarly as in the derivation of the MCS formulation \eqref{eq::mixedstressstokesweak}, motivated by rewriting the fourth order problem of the stream function formulation \eqref{eq:stream} as a mixed system. To his end we introduce the space
\begin{align*}
  \Sigma^\hhj :=\{\sigma \in L^2(\Omega, \rr^{d\times d}): \sigma = \sigma^\trans, \sigma|_K \in H^1(T, \rr^{d \times d}) ~\forall T \in \mesh, \jump{\sigma_{nn}} = 0 ~\forall F \in \facets \},
\end{align*}
and define the symmetric tensor $\sigma := \nabla^2\psi \in \Sigma^\hhj$  where $\psi \in \Psi$ is the stream function. Then, with the introduction of the symmetric bilinear form
\begin{align*}
  a(\sigma,\tau) :=&  \frac{1}{\nu} \int_\Omega \sigma : \tau \dif x,
\end{align*}
and the bilinear form 
\begin{align}
  b^\hhj(\sigma,\phi) :=& -\sum\limits_{T \in \mesh} \int_T \sigma : \nabla^2\phi \dif x + \sum\limits_{F \in \facets }\int_{F} \sigma_{nn} \jump{(\nabla\phi)_n} \dif s \label{eq::hhjbblf} \\
   =& \sum\limits_{T \in \mesh} \int_T \divergence(\sigma) \cdot \nabla\phi \dif x - \sum\limits_{F \in \facets }\int_{F} \jump{\sigma_{nt}} \cdot (\nabla\phi)_t \dif s, \nonumber
\end{align}
an equivalent formulation of equation~\eqref{eq:stream} is given by: Find $(\sigma,\psi) \in \Sigma \times \Psi$ such that
%\begin{subequations}
  \begin{alignat*}{2}
    a(\sigma,\tau) + b^\hhj(\tau,\psi) & = 0 &&    \quad \forall \tau \in \Sigma^\hhj, \\%\label{eq::conthhjone}  \\
    b^\hhj(\sigma,\phi) & = -\int_\Omega  f \cdot \curl(\phi) \dif x    &&    \quad \forall \phi \in \Psi. % \label{eq::conthhjtwo}
\end{alignat*}
%\label{eq::conthhj} 
%\end{subequations}
In order to derive the discrete HHJ method we now define the approximation spaces
\begin{align} \label{def::dischhjspaces}
  S^k_h := \mathcal{P}^k(\mesh, \rr) \cap S  \quad \textrm{and} \quad  \Sigma^{\hhj, k-1}_h := \mathcal{P}^{k-1}(\mesh, \rr^{d \times d}) \cap \Sigma^\hhj,
\end{align}
then the HHJ method is given by: Find $(\sigma_h,\psi_h) \in \Sigma_h^{\hhj, k-1} \times S^k_h$ such that
\begin{subequations}
  \begin{alignat}{2}
    a(\sigma_h,\tau_h) + b^\hhj(\tau_h,\psi_h) & = 0 &&    \quad \forall \tau_h \in \Sigma_h^{\hhj,k-1}, \label{eq::dischhjone}  \\
    b^\hhj(\sigma_h,\phi_h) & = -\int_\Omega  f \cdot \curl(\phi_h) \dif x    &&    \quad \forall \phi_h \in S^k_h.  \label{eq::dischhjtwo}
\end{alignat}
\label{eq::dischhj} 
\end{subequations}
Whereas this method gives optimal convergence orders, it is not clear how it can be extended to the three dimensional case. As stated in \cite{Girault:book}, \textit{``the obvious reason is, that the conditions determining the vector potential are more intricate than those defining the two-dimensional stream function''}. %A detailed discussion on this topic given in section \ref{sec::relation}.
\subsection{A HHJ-like method}
We now introduce a new discrete method for the discretization of equation \eqref{eq::mixedstressstreamweak} and \eqref{eq::mixedstressstreamweaktwod}.

We start with the case $d = 2$. In contrast to the previous section where $\sigma^\hhj$ was the symmetric hessian of $\psi$, we now aim to approximate the matrix $\sigma = "\nabla \curl(\psi)"$. For the approximation of $\sigma \in \Sigma \subset H(\curl \div)$ we follow the works \cite{Lederer:2019b, Lederer:2019c, lederer2019mass} and define the discrete stress space
\begin{align*}
  \Sigma^k_h := \{ \sigma_h \in \mathcal{P}^k(\mesh, \rr^{d \times d}): \tr(\sigma_h) = 0, &\jump{(\sigma_h)_{nt}} = 0, \\ &~~(\sigma_h)_{nt} \in \mathcal{P}^{k-1}(F, \rr^{d-1})   ~ \forall F \in \facets \}.
\end{align*}
Note, that the discrete space $\Sigma^k_h$ is slightly non-conforming with respect to $H(\curl \divergence)$. Now let $S^k_h$ be as in \eqref{def::dischhjspaces} and define the velocity space as
\begin{align*}
  V^k_h := \mathcal{P}^k(\Omega, \rr^d) \cap H_0(\divergence, \Omega),
\end{align*}
which is the well known $H(\divergence)$-conforming Brezzi-Douglas-Marini space, see for example in \cite{Boffi:book}. In \cite{Lederer:2019b,lederer2019mass} the authors motivated the definition of a discrete duality bracket which mimics $\langle \divergence(\sigma), v \rangle_{H_0(\divergence)}$ in the case where $\sigma \in \Sigma^k_h \not \subset H(\curl \divergence)$. Following these ideas, we define for all $\sigma_h \in \Sigma^k_h$ and all $v_h \in V^k_h$ the bilinear form
\begin{align*}
  b(\sigma_h, v_h) := & -\sum\limits_{T \in \mesh} \int_T \sigma : \nabla v_h \dif x + \sum\limits_{F \in \facets }\int_{F} (\sigma_h)_{nt} \jump{(v_h)_t} \dif s \\
  =& \sum\limits_{T \in \mesh} \int_T \divergence(\sigma_h) \cdot v_h \dif x - \sum\limits_{F \in \facets }\int_{F} \jump{\sigma_{nn}} \cdot (v_h)_n \dif s.
\end{align*}
Similarly as in the continuous setting, the properties of the discrete de Rham complex now give $\curl(S^k_h) \subset V^{k-1}_h$. Thus for a function $\phi_h \in S^k_h$ and a $\sigma_h \in \Sigma^k_h$, the discrete duality bracket then reads as
\begin{align} \label{eq:discbblf}
  b(\sigma_h, \curl (\phi_h)) := & -\sum\limits_{T \in \mesh} \int_T \sigma_h : \nabla \curl (\phi_h) \dif x + \sum\limits_{F \in \facets }\int_{F} (\sigma_h)_{nt} \jump{\curl (\phi_h)_t} \dif s \\
 =&  \sum\limits_{T \in \mesh} \int_T \div(\sigma_h) \cdot \curl (\phi_h) \dif x - \sum\limits_{F \in \facets }\int_{F} \jump{(\sigma_h)_{nn}} \curl( \phi_h)_n \dif s.
\end{align}
Comparing this to the definition of the bilinear form $b^\hhj(\cdot, \cdot)$, see equation \eqref{eq::hhjbblf}, we realize that $b(\cdot, \cdot)$ reads, simply said, as a rotated version: Whereas a $\sigma_h^{\hhj,k} \in \Sigma_h^\hhj$ is ``normal-normal'' continuous and $\nabla \phi_h$ is tangential continuous, we are now in the setting where $\sigma_h \in \Sigma^k_h$ is ``normal-tangential'' continuous and $\curl(\phi_h)$ is normal continuous. With respect to the weak formulation \eqref{eq::mixedstressstreamweaktwod} we now define the discrete method: Find $(\sigma_h,\psi_h) \in \Sigma^{k-1}_h \times S^k_h$ such that
\begin{subequations}
  \begin{alignat}{2}
    a(\sigma_h,\tau_h) + b(\tau_h, \curl (\psi_h)) & = 0 &&    \quad \forall \tau_h \in \Sigma^{k-1}_h, \label{eq::dischhjrotone}  \\
    b(\sigma_h,\curl( \phi_h)) & = -\int_\Omega  f \cdot \curl(\phi_h) \dif x    &&    \quad \forall \phi_h \in S^k_h.  \label{eq::dischhjrottwo}
\end{alignat}
\label{eq::dischhjrot} 
\end{subequations}
\begin{remark}
Note, that on one edge $F \in \facets$ the normal-tangential components of functions in $\Sigma^{k-1}_h$ are polynomials of order $k-2$, whereas the normal-normal components of functions in $\Sigma^{\hhj, k-1}_h$ are polynomials of order $k-1$ resulting in less coupling degrees of freedom.
\end{remark}
Considering the close relation of equation \eqref{eq::mixedstressstreamweak} and equation \eqref{eq::mixedstressstreamweaktwod} in the continuous setting, the derivation of the two dimensional (rotated) HHJ-like method motivates to extend it also to the three dimensional case. Following \cite{Boffi:book}, the tangential continuous N\'{e}d\'{e}lec finite element space is given by
\begin{align*}
  W^k_h := \mathcal{P}^k(\Omega, \rr^d) \cap H_0(\curl, \Omega),
\end{align*}
with the property $\curl(W^k_h) \subset V^{k-1}_h$. Similarly as in the continuous setting, let the gauging bilinear form be given by
\begin{align*}
  b^g(\phi_h, \lambda_h) = \int_\Omega \phi_h \cdot \nabla \lambda_h \dif x \quad \forall \phi_h \in W^k_h, \forall \lambda_h \in S^{k+1}_h.
\end{align*}
Then, the three dimensional modified HHJ method is given by: Find $(\sigma_h, \psi_h, \lambda_h) \in \Sigma^{k-1}_h \times W^k_h \times S^{k+1}_h$ such that
\begin{subequations}
  \begin{alignat}{2}
    a(\sigma_h,\tau_h) + b(\tau_h,\curl( \psi_h)) & = 0 &&    \quad \forall \tau_h \in \Sigma^{k-1}_h,\\
    b(\sigma_h,\curl( \phi_h)) + b^g(\phi_h, \lambda_h) & = -\int_\Omega  f \cdot \curl(\phi_h) \dif x    &&    \quad \forall \phi_h \in W^k_h, \\
b^g(\psi_h, \mu_h) & = 0    &&    \quad \forall \mu_h \in S^{k+1}_h.
\end{alignat}
\label{eq::dischhjrotthree} 
\end{subequations}
\begin{remark}
  Whereas the discrete velocity of a standard mixed finite element approximation of the Stokes equations might only be discretely divergence-free, the velocity $u_h := \curl (\psi_h)$, where $\psi_h$ is the solution of \eqref{eq::dischhjrotthree}, is exactly divergence-free, see also Remark \ref{rem::pressurerob}.
\end{remark}
\begin{remark} \label{rem::nograds}
  Testing the second equation of \eqref{eq::dischhjrotthree} with the test function $\phi_h = \nabla \lambda_h$ shows, that the Lagrangian multiplier $\lambda_h$ equals zero, thus is only needed to prove discrete stability of the system. As discussed in the next section, discrete stability can also be proven on the space
  \begin{align} \label{eq::orthcomp}
  W^{k,\perp}_h := \{ \phi_h \in W^k_h: \int_\Omega \phi_h \cdot \nabla \mu_h \dif x = 0~ \forall \mu_h \in S^{k+1}_h \}.
  \end{align}
  Then we have the problem: Find $(\sigma_h, \psi_h) \in \Sigma^{k-1}_h \times W_h^{k,\perp}$ such that
\begin{subequations}
  \begin{alignat}{2}
    a(\sigma_h,\tau_h) + b(\tau_h,\curl( \psi_h)) & = 0 &&    \quad \forall \tau_h \in \Sigma^{k-1}_h,\\
    b(\sigma_h,\curl (\phi_h)) & = -\int_\Omega  f \cdot \curl(\phi_h) \dif x    &&    \quad \forall \phi_h \in W_h^{k,\perp}.
\end{alignat}
\label{eq::dischhjrotthreenograd}
\end{subequations}
Note, that the solutions $\sigma_h$ and $\psi_h$ of \eqref{eq::dischhjrotthree} and \eqref{eq::dischhjrotthreenograd} are identical. For the implementation of \eqref{eq::dischhjrotthreenograd}, the corresponding finite element code needs a basis of $W^{k,\perp}_h$. This can be achieved if the finite element spaces are constructed with respect to the discrete de Rham complex. For the high order moments see for example in \cite{zaglmayr2006high}. We also want to mention, that $W^{k,\perp}_h$ has less coupling degrees of freedoms compared to $W^k_h$, thus the factorization step to solve the corresponding linear system is faster.
\end{remark}
\section{A stability analysis in mesh dependent norms}
In this section we present a stability and error analysis of the HHJ-like method introduced in the last chapter. The analysis is based on using mesh dependent norms as for example in \cite{babuska} and \cite{Stenberg1988}. We only proof the three dimensional case since it is more challenging due to the gauging bilinear form. The two dimensional case follows with similar techniques.

In contrast to the continuous stability analysis of Section \ref{sec::redregstreamform}, we aim to use the $L^2$-norm on the discrete space $\Sigma^k_h$ and use again $\| \nabla \cdot \|_{L^2}$ on $S^k_h$. For the stream function and the velocity space we define for all $\phi_h \in W^k_h$ and $u_h \in V^k_h$ the norms
\begin{align*}
  | \phi_h |^2_{1,\curl,h} &:= \sum\limits_{T \in \mesh} \|\nabla \curl(\phi_h)\|_T^2 + \sum\limits_{F \in \facets} \frac{1}{h} \| \jump{  \curl(\phi_h)_t }\|_{F}^2, \\
  \| \phi_h \|^2_{1,\curl,h} &:= \| \phi_h \|^2_{L^2}+| \phi_h |^2_{1,\curl,h}, \\
    \| u_h \|^2_{1,h} &:= \sum\limits_{T \in \mesh} \|\nabla u_h\|_T^2 + \sum\limits_{F \in \facets} \frac{1}{h} \| \jump{  (u_h)_t }\|_{F}^2.
\end{align*}
Note, that the norm $\| \phi_h \|^2_{1,\curl,h}$ reads as a discrete $H^1(\curl)$-like norm, and that for $u_h = \curl(\psi_h)$ we have $| \psi_h |_{1,\curl,h} = \| u_h \|_{1,h}$. For the stability proof we will need the following decomposition and norm equivalence results.
\begin{lemma} \label{lem::normeqsigma}
  Let $\sigma_h \in \Sigma_h^{k-1}$ be arbitrary. There holds the norm equivalence
  \begin{align*}
    \| \sigma_h \|_{L^2}^2 \sim \sum\limits_{T \in \mesh} \| \sigma_h \|_T^2 + \sum\limits_{F \in \facets} h \| (\sigma_h)_{nt} \|_F^2.
  \end{align*}
\end{lemma}
\begin{proof}
This follows by standard scaling arguments and was proven in \cite{Lederer:2019b}.
\end{proof}
\begin{lemma} \label{lem::normeqphi}
  Let $\phi_h \in W^k_h$, then there exists a $\lambda_h \in S^{k+1}_h$ and a $w_h \in W^k_h$ such that $\phi_h = \nabla \lambda_h + w_h$ and
  \begin{align} \label{eq::decompest}
    \| \nabla \lambda_h \|_{L^2} \lesssim \| \phi_h \|_{L^2} \quad \textrm{and} \quad \| w_h \|_{L^2} \lesssim \| \curl(\phi_h) \|_{L^2}.
  \end{align}
  Further, there holds the norm equivalence
  \begin{align} \label{eq::discnormequi}
    | \phi_h |^2_{1,\curl,h} + \| \nabla \lambda_h \|_{L^2} \sim \| \phi_h \|^2_{1,\curl,h}.
  \end{align}
\end{lemma}
\begin{proof}
  The first part of the lemma is well known in the literature but the proof is presented for completeness. First, we solve the problem: Find $\lambda_h \in S^{k+1}_h$ such that
  \begin{align*} 
    \int_\Omega \nabla \lambda_h \cdot \nabla \mu_h \dif x = \int_\Omega \phi_h \cdot \nabla \mu_h \dif x \quad \forall \mu_h \in S^{k+1}_h.
  \end{align*}
  Solveability of this problem is given by the standard theory, see for example \cite{brenner:book}, and there holds the regularity estimate $\| \nabla \lambda_h \|_{L^2} \lesssim \| \phi_h \|_{L^2}$. Now, since $w_h := \phi_h -\nabla \lambda_h \in W^k_h$ is $L^2$-orthogonal on $\nabla S^{k+1}_h$ (by definition), the estimate $\| w_h \|_{L^2} \lesssim \| \curl(\phi_h) \|_{L^2}$ follows by $\curl(\nabla \lambda_h) = 0$ and a Friedrichs-type inequlity for the $H(\curl)$-space, see for example \cite{monk2003finite}.

  We continue with the proof of \eqref{eq::discnormequi}. The left side can be bounded by the right side by applying the estimate for $\lambda_h$ in \eqref{eq::decompest}. For the other direction, the triangle inequality, and the right estimate in \eqref{eq::decompest} gives
  \begin{align*}
    \| \phi_h \|_{L^2}^2 &\le \| \nabla \lambda_h \|_{L^2}^2 + \| w_h \|_{L^2}^2 \lesssim  \| \nabla \lambda_h \|_{L^2}^2 + \| \curl(\phi_h) \|_{L^2}^2.
  \end{align*}
  With $u_h := \curl(\phi_h) \in V^{k-1}_h$, a discrete Friedrichs-like inequality on $V^{k-1}_h$, as in \cite{brenner_friedrich}, gives
  \begin{align*}
    \| \curl(\phi_h) \|_{L^2}^2 = \| u_h \|_{L^2}^2 \lesssim \| u_h \|_{1,h}^2 = \| \phi_h \|_{1,\curl,h}^2,
  \end{align*}
  from what we conclude the proof.
\end{proof}
To show discrete stability, we again apply the standard theory of mixed problems, see \cite{Boffi:book}. To this end we prove continuity, kernel ellipticity and inf-sup stability in the following.
\begin{lemma}[Continuity] \label{lem::disccont}
  There holds the continuity estimate
  \begin{alignat*}{2}
    a(\sigma_h,\tau_h) &\lesssim \frac{1}{\nu} \| \sigma_h\|_{L^2} \| \tau_h\|_{L^2} &&\quad \forall \sigma_h, \tau_h \in \Sigma^{k-1}_h \\
    b(\sigma_h,\curl (\phi_h)) &\lesssim \| \sigma_h\|_{L^2} \| \phi_h\|_{1,\curl,h} &&\quad \forall \sigma_h \in \Sigma^{k-1}_h , \forall\phi_h \in W^k_h \\
    b^g(\phi_h, \mu_h) &\lesssim  \| \phi_h\|_{1,\curl,h} \| \nabla \mu_h\|_{L^2} &&\quad \forall \phi_h \in W^{k}_h, \forall\mu_h \in S^{k+1}_h.
  \end{alignat*}
\end{lemma}
\begin{proof}
The continuity of $a(\cdot,\cdot)$ and $b^g(\cdot,\cdot)$ follows by the Cauchy-schwarz inequality. For $b(\cdot,\cdot)$ we use the first representation in \eqref{eq:discbblf}, the Cauchy-Schwarz inequality on each element $T$ and $F$ separately, and the norm equivalence Lemma \ref{lem::normeqsigma}.
\end{proof}
\begin{lemma}[Kernel ellipticity]\label{lem::discell}
  Let $(\sigma_h,\lambda_h) \in \Sigma^{k-1}_h \times S^{k+1}_h$ be an element in the kernel of the constraints, i.e. $b( \sigma_h, \curl( \phi_h)) + b^g(\phi_h, \lambda_h) = 0$ for all $\phi_h \in W^k_h$. There holds the estimate
  \begin{align*}
    \| \sigma_h \|^2_{L^2} + \|\nabla \lambda_h \|^2_{L^2} \lesssim \nu a(\sigma,\sigma).
  \end{align*}
\end{lemma}
\begin{proof}
  Since $\| \sigma_h \|^2_{L^2} = \nu a(\sigma,\sigma)$, we only prove the estimate for $\lambda_h$. 
  Now let $\phi_h:=\nabla \lambda_h \in W^k_h$. As $\curl (\phi_h) = 0$, the definition of the bilinear form $b(\cdot, \cdot)$ gives
  \begin{align*}
    \| \nabla \lambda_h \|^2_{L^2} & =  \int_\Omega \nabla \lambda_h \cdot \nabla \lambda_h \dif x = \int_\Omega \phi_h \cdot \nabla \lambda_h \dif x  \\
    & =   \sum\limits_{T \in \mesh} \int_T \sigma_h : \nabla \curl (\phi_h)  \dif x - \sum\limits_{F \in \facets }\int_{F} (\sigma_h)_{nt} \jump{\curl (\phi_h )_t} \dif s = 0,
  \end{align*}
which implies (due to the boundary conditions) that $\lambda = 0$, and the lemma is proven.
\end{proof}
\begin{lemma}[inf-sup]\label{lem::discinfsup}
  There holds the stability estimate
  \begin{align*}
    \sup \limits_{\substack{0 \neq\sigma_h \in \Sigma^{k-1}_h \\ 0 \neq \lambda_h \in S^{k+1}_h}} \frac{b( \sigma_h, \curl \phi_h) + b^g(\phi_h, \lambda_h)}{\|\sigma_h \|_{L^2} + \| \nabla \lambda_h \|_{L^2}} \gtrsim \|\phi_h \|_{1,\curl,h} \quad \forall \phi_h \in W^k_h.
  \end{align*}
  
\end{lemma}
\begin{proof}
  Let $\phi_h \in W^k_h$ be arbitrary and set $u_h := \curl( \phi_h) \in V^{k-1}_h$. Since $\div u_h = 0$, Lemma 6.5 (and the norm equivalence (6.5)) in \cite{Lederer:2019b} shows, that there exists a function $\sigma_h \in \Sigma^{k-1}_h$ such that $\| \sigma_h \|_{L^2} \lesssim \|u_h \|_{1,h}$ and $b(\sigma,u_h) \ge c_1 \|u_h\|^2_{1,h}$, where $c_1>0$ is a fixed constant. Now let $\lambda_h \in S^{k+1}_h$ and $w_h \in W^k_h$ be the decomposition functions given by Lemma  \ref{lem::normeqphi}, such that $\| w_h \|_{L^2} \le c_2 \| \curl (\phi_h) \|_{L^2}$. As in the proof of Lemma~\ref{lem::normeqphi}, a Friedrichs-like inequality further gives $ \| u_h \|_{L^2} \le c_3 \|u_h\|_{1,h}$, thus 
  \begin{align} \label{whestimate}
    \| w_h \|_{L^2} \le c_2 \| \curl \phi_h \|_{L^2} \le c_2c_3 \|u_h\|_{1,h}.
  \end{align}
  Now with $\tilde \sigma_h := \frac{c_2c_3}{c_1} \sigma_h$ we get
  \begin{align*}
    b( \tilde \sigma_h, \curl (\phi_h)) + b^g(\phi_h, \lambda_h) &= b( \tilde \sigma_h, u_h) + \int_{\Omega} \phi_h \cdot \nabla \lambda_h \dif x \\
                                                        &= b(\tilde  \sigma_h, u_h) + \int_{\Omega} |\nabla \lambda_h|^2 - w_h \cdot \nabla \lambda_h \dif x \\
                                                               &\ge c_2 c_4\|u_h\|^2_{1,h} + \|\nabla \lambda_h\|_{L^2}^2 - \|w_h\|_{L^2}\|\nabla \lambda_h\|_{L^2}
  \end{align*}
  Using \eqref{whestimate}, Youngs inequality and $\| u_h \|_{1,h} = |\phi_h|_{1,\curl,h}$, finally gives
  \begin{align*}  
      b( \tilde \sigma_h, \curl (\phi_h)) + b^g(\phi_h, \lambda_h) &\ge c_2 c_4 \|u_h\|^2_{1,h} + \|\nabla \lambda_h\|_{L^2}^2 - c_2c_4 \|u_h\|_{1,h}\|\nabla \lambda_h\|_{L^2}\\
                                                               &\gtrsim |\phi_h|^2_{1,\curl,h} + \|\nabla \lambda_h\|_{L^2}^2.
  \end{align*}
With the continuity estimates $\| \sigma_h \|_{L^2} \lesssim |\phi_h|_{1,\curl,h}$, and the left estimate of equation~\eqref{eq::decompest}, we conclude the proof by the norm equivalence~\eqref{eq::discnormequi}.
\end{proof}
The above prove and the norm equivalence \eqref{eq::discnormequi} also shows, that we can derive an inf-sup condition on the orthogonal complement of $\nabla S^{k+1}_h$, i.e. there holds
\begin{align} \label{es::infsupsubset}
    \sup \limits_{0 \neq\sigma_h \in \Sigma^{k-1}_h} \frac{b( \sigma_h, \curl (\phi_h))}{\|\sigma_h \|_{L^2}} \gtrsim \|\phi_h \|_{1,\curl,h} \quad \forall \phi_h \in W^{k,\perp}_h.
\end{align}
see also Remark \ref{rem::nograds}.
\begin{theorem}[Consistency] \label{th::consistency}
  The HHJ-like method for the stream function is consistent in the following sense. If the exact solution of the Stokes problem \eqref{eq:mixedstokes} is such that $u\in H^1(\Omega, \rr^d)$, $\sigma \in H^1(\Omega, \rr^{d \times d})$, $p \in L^2_0$, and $\psi \in \Psi$ is the exact stream function such that $\curl(\psi) = u$, then
  \begin{align*}
    a(\sigma, \tau_h) + b(\tau_h, \curl (\psi) ) + b(\sigma, \curl ( \phi_h)) = -\int_\Omega f \cdot \curl ( \phi_h) \dif x,
  \end{align*}
  for all $\tau_h \in \Sigma^{k-1}_h, \phi_h \in W^k_h$.
\end{theorem}
\begin{proof}
  As the exact solutions $\sigma$ and $u = \curl (\psi)$ are continuous we have that $\jump{\sigma_{nn}} = 0$ and $\jump{\curl (\psi)_t} = 0$ on all faces $F \in \facets$, thus by definition \eqref{eq:discbblf}, we have
  \begin{align*}
    b(\sigma, \curl( \phi_h)) &= \sum\limits_{T \in \mesh} \int_T \div(\sigma) \cdot \curl (\phi_h) \dif x - \sum\limits_{F \in \facets }\int_{F} \jump{\sigma_{nn}} \curl (\phi_h)_n \dif s  \\
                            &= \int_\Omega \divergence(\sigma) \cdot \curl( \phi_h) \dif x, \\
      b(\tau_h, \curl (\psi)) = &- \sum\limits_{T \in \mesh} \int_T \tau_h : \nabla \curl (\psi) \dif x + \sum\limits_{F \in \facets }\int_{F} (\tau_h)_{nt} \jump{\curl (\psi)_t} \dif s \\
&=  -\int_\Omega \tau_h : \nabla \curl (\psi) \dif x.
  \end{align*}
  Now, since $\dev{\sigma} = \sigma = \nu \nabla u = \nu \nabla \curl (\psi)$ we have
  \begin{align*}
    a(\sigma, \tau_h) + b(\tau_h, \curl (\psi) ) =  \int_\Omega \frac{1}{\nu} \sigma : \tau_h \dif x -\int_\Omega \tau_h : \nabla \curl (\psi) \dif x = 0.
  \end{align*}
  With $\divergence(\sigma) = - f + \nabla p$, integration by parts for the pressure integral finally  gives
  \begin{align*}
     b(\sigma, \curl (\phi_h)) = \int_\Omega (-f + \nabla p) \cdot \curl(\phi_h) \dif x = -\int_\Omega f \cdot \curl (\phi_h) \dif x,
  \end{align*}
  where we used that $\curl(\nabla p) = 0$.
\end{proof}
\subsection{An error  analysis of the HHJ-like method} \label{sec::errorana}
The a priori error estimate presented in this section is based on the inf-sup stability and the consistency proven before. Further we need several interpolation results.

To this end let $I_{V^{k-1}_h}$ and $I_{W^k_h}$ be an $H(\divergence)$-conforming and $H(\curl)$-conforming (projection based) interpolation operator, respectively, as for example in \cite{MR1746160, Demkowicz2008}. Note, that these operators commute with the corresponding differential operators, i.e.
\begin{align} \label{eq::commintop}
  I_{V^{k-1}_h} \curl = \curl I_{W^k_h},
\end{align}
Further, let $I_{\Sigma^{k-1}_h}$ be the interpolation operator defined in \cite{Lederer:2019b}.
\begin{lemma}\label{lem::intop}
Let $u, \sigma$ be arbitrary with $u \in H^1(\Omega, \rr^d) \cap H^m(\mathcal{T}, \rr^d)$ and $\sigma \in H^1(\Omega, \rr^{d\times d}) \cap H^{m-1}(\mesh,\Omega^{d \times d})$. For $s = \min(m-1, k-1)$, there holds the approximation estimate
\begin{align*}
  \| \sigma - I_{\Sigma^{k-1}_h} \sigma \|_{L^2} + \sqrt{ \sum\limits_{F \in \facets} h \|( \sigma - I_{\Sigma^{k-1}_h} \sigma )_{nt} \|^2_{F}} &\lesssim h^s \| \sigma \|_{H^{s}(\mesh)},\\
  \| u - I_{V^{k-1}_h} u \|_{1,h} &\lesssim h^s \| u \|_{H^{s+1}(\mesh)}.
\end{align*}
\end{lemma}
\begin{proof}
The estimate of $I_{V^{k-1}_h}$ follows with the standard techniques and is based on the Bramble-Hilbert lemma. We refer to for example to \cite{lehrenfeld2010hybrid, Boffi:book} for a detailed proof. The proof for $I_{\Sigma^{k-1}_h}$ can be found in \cite{Lederer:2019b}. Note, that the estimate for $I_{\Sigma^{k-1}_h}$ is motivated by the norm equivalence given in Lemma \ref{lem::normeqsigma}.
\end{proof}
\begin{theorem}[Optimal convergence] \label{th::convergence}
  Let $u \in H^1(\Omega, \rr^d) \cap H^m(\mathcal{T}, \rr^d)$, $\sigma \in H^1(\Omega, \rr^{d\times d}) \cap H^{m-1}(\mesh,\Omega^{d \times d})$ be the exact solution of  \eqref{eq:mixedstokes}, and let $\psi \in \Psi$ be the exact stream function such that $\curl(\psi) = u$. Let $\sigma_h, \psi_h \in \sigma_h^{k-1} \times W_h^k$ be the solution of the HHJ-like method \eqref{eq::dischhjrotthree} and set $u_h := \curl(\psi_h)$. For $s = \min(m-1, k-1)$ there holds the error estimate
  \begin{align*}
    \frac{1}{\nu}\| \sigma - \sigma_h\|_{L^2} + \| u - u_h \|_{1,h} \lesssim h^s \| u \|_{H^{s+1}(\mesh)}
  \end{align*}
\end{theorem}
\begin{proof}
  In a first step we bound the error by the triangle inequality which gives
  \begin{align*}
    \frac{1}{\nu}\| \sigma - \sigma_h\|_{L^2} + \| u - u_h \|_{1,h} &\le \frac{1}{\nu}\| \sigma - I_{\Sigma^{k-1}_h} \sigma\|_{L^2} + \| u - I_{V^{k-1}_h}u \|_{1,h} \\
    &+ \frac{1}{\nu}\|  I_{\Sigma^{k-1}_h} \sigma - \sigma_h\|_{L^2} + \| I_{V^{k-1}_h}u - u_h \|_{1,h}.
  \end{align*}
  Applying Lemma \ref{lem::intop} shows, that the first two terms on the right side already converge with the optimal order. We continue with the last two terms on the right side. Since, $u_h = \curl \psi_h$ and $u = \curl \psi$, the commuting property of the interpolation operators, see equation \eqref{eq::commintop}, gives
  \begin{align*}
    \frac{1}{\nu}\|  I_{\Sigma^{k-1}_h} \sigma - \sigma_h\|_{L^2} +  \| I_{V^{k-1}_h}u &- u_h \|_{1,h} \\
    &= \frac{1}{\nu}\|  I_{\Sigma^{k-1}_h} \sigma - \sigma_h\|_{L^2} +  \| I_{V^{k-1}_h}\curl \psi - \curl \psi_h \|_{1,h} \\
    &= \frac{1}{\nu}\|  I_{\Sigma^{k-1}_h} \sigma - \sigma_h\|_{L^2} +  \| \curl I_{W^k_h}\psi - \curl \psi_h \|_{1,h} \\
   &\le \frac{1}{\nu}\|  I_{\Sigma^{k-1}_h} \sigma - \sigma_h\|_{L^2} +  \| I_{W^k_h}\psi - \psi_h \|_{1,\curl,h}.
  \end{align*}
  In the following we aim to use the discrete stability results proven in the last section. To this end we define the following product space norm
  \begin{align*}
    \| (\tau_h,\phi_h)\|_* := \sqrt{\nu} \| \phi_h\|_{1,\curl,h} + \frac{1}{\sqrt{\nu}}\|\tau_h\|_{L^2} \quad \forall (\tau_h,\phi_h) \in \Sigma^{k-1}_h \times W^k_h.
  \end{align*}
  Following the definition of $I_{W^k_h}$, equation (200) in \cite{Demkowicz2008}, we see that for an arbitrary $\mu_h \in S^{k+1}_h$ we have 
  \begin{align*}
    \int_\Omega I_{W^{k}_h} \psi \cdot \nabla \mu_h \dif x =  \int_\Omega \psi \cdot \nabla \mu_h \dif x = - \int_\Omega  \div(\psi) \mu_h \dif x = 0,
  \end{align*}
  where we used that the stream function fulfills $\div(\psi) = 0$. Now as $\psi_h$ is the solution of \eqref{eq::dischhjrotthree}, we have $I_{W^k_h}\psi - \psi_h \in W_h^{k,\perp}$, thus Lemma \ref{lem::disccont}, Lemma \ref{lem::discell} and the inf-sup condition \eqref{es::infsupsubset} gives the estimate 
  \begin{align*}
    \|  &(I_{\Sigma^{k-1}_h} \sigma - \sigma_h, I_{W^k_h}\psi - \psi_h)\|_{*}\\
        & \lesssim \!\!\!\sup\limits_{\substack{(\tau_h,\phi_h) \in \\  \Sigma^{k-1}_h \times W^{k,\perp}_h}}\!\!\! \frac{a(I_{\Sigma^{k-1}_h} \sigma - \sigma_h, \tau_h) + b(I_{\Sigma^{k-1}_h} \sigma - \sigma_h,\curl (\phi_h)) + b(\tau_h,\curl( I_{W^k_h}\psi - \psi_h)) }{\| ( \tau_h, \phi_h) \|_{*}} \\
    & \lesssim \!\!\!\sup\limits_{\substack{(\tau_h,\phi_h) \in \\  \Sigma^{k-1}_h \times W^{k,\perp}_h}}\!\!\! \frac{a(I_{\Sigma^{k-1}_h} \sigma - \sigma, \tau_h) + b(I_{\Sigma^{k-1}_h} \sigma - \sigma,\curl (\phi_h)) + b(\tau_h,\curl( I_{W^k_h}\psi - \psi)) }{\| ( \tau_h, \phi_h) \|_{*}}.
  \end{align*}
  where we used Theorem \ref{th::consistency} in the last step. Following the same steps as in the proof of Theorem 6.3 in \cite{Lederer:2019b}, the Cauchy-Schwarz inequality (as in the proof of Lemma \ref{lem::disccont}) allows us to bound
  \begin{align*}
    &a(I_{\Sigma^{k-1}_h} \sigma - \sigma, \tau_h) + b(I_{\Sigma^{k-1}_h} \sigma - \sigma,\curl \phi_h) + b(\tau_h,\curl( I_{W^k_h}\psi - \psi)) \\
    & \lesssim  \left(\| ( I_{\Sigma^{k-1}_h} \sigma - \sigma, I_{W^k_h}\psi - \psi)\|_{*} + \frac{1}{\sqrt{\nu}}\sqrt{ \sum\limits_{F \in \facets} h \|( I_{\Sigma^{k-1}_h} \sigma - \sigma)_{nt} \|^2_{F}} \right)\| ( \tau_h, \phi_h) \|_{*},
  \end{align*}
and we conclude the proof with the interpolation error estimates of Lemma \ref{lem::intop}.
\end{proof}
\begin{remark} \label{rem::optconv}
  We want to emphasize that although $s = \min(m-1,k-1)$, the result of Theorem \ref{th::convergence} reads as an optimal convergence result. Since the fixed polynomial order $k$ corresponds to the approximation order of the stream function $\psi_h$ in the space $W_h^k$, it follows that for $u_h = \curl \psi_h \in V_h^{k-1}$ the convergence rate of the error measured in a discrete $H^1$-like norm is only expected to be at most of order $\mathcal{O}(h^{k-1})$.
  Further note, that if the finite element library for the implementation allows an approximation of the reduced system \eqref{eq::dischhjrotthreenograd}, hence provides an explicit basis for $W_h^{k,\perp}$, the method is also optimal with respect to the number of degrees of freedom compared to a direct approximation of $u_h \in V_h^{k-1}$ as for example in \cite{Lederer:2019b,Lehrenfeld:2016}. This follows directly by the properties of the discrete de Rham complex as in \cite{zaglmayr2006high}.
\end{remark}
\begin{remark}[Pressure robustness]\label{rem::pressurerob}
 Theorem \ref{th::convergence} shows that the velocity error can be bounded independently of the continuous pressure solution. Methods that allow to deduce such error estimates are called pressure robust and we present a brief explanation in the following. Let $\mathbb{P}$ be the continuous Helmholtz projection (see \cite{Girault:book}) onto the rotational part of a given load $f$
\begin{align*}
  f = \nabla \theta + \xi =: \nabla \theta + \mathbb{P}(f),
\end{align*}
with $\theta \in H^1(\Omega)/\rr$ and  $\xi =: \mathbb{P}(f) \in \{ v \in H_0(\divergence, \Omega): \divergence(v) = 0\}$. Testing the momentum balance \eqref{eq:mixedstokestwo} with an arbitrary (exactly) divergence-free test function $v$, shows that $\sigma = \nu \nabla u$ is steered only by $\mathbb{P}(f)$ since integration by parts gives
\begin{align*}
  \int_\Omega \nabla \theta \cdot v \dif x = 0.
\end{align*}
In \cite{Linke:2014} the author showed that this property might not be
handed over from the continuous to the discrete setting since for classical mixed methods 
a discrete divergence-free velocity test function might not be exactly divergence-free. In this case, one can only deduce a velocity error estimate that depends on the best approximation of the continuous pressure solution that includes a scaling $1/\nu$ which can produce big errors for vanishing viscosities $\nu \rightarrow 0$, see \cite{Linke:2016a,Linke:2016c, Lederer:2017b, John:2017}. One advantage of the stream function formulation is, that the right hand side of the weak formulation, see \eqref{eq:stream}, \eqref{eq::mixedstressstreamweak} and \eqref{eq::dischhjrotthree}, is only tested with $\curl(\phi)$ or $\curl(\phi_h)$ in the continuous and discrete setting, respectively. Hence, in both situations we have again with integration by parts that
\begin{align*}
  \int_\Omega \nabla \theta \cdot \curl(\phi) \dif x =   \int_\Omega \nabla \theta \cdot \curl(\phi_h) \dif x = 0,
\end{align*}
and thus the (discrete) velocity $u_h = \curl(\psi_h)$ is again only steered by $\mathbb{P}(f)$ which allows to derive the pressure robust error estimate of Theorem \ref{th::convergence}.
\end{remark}
\section{Post processing for the pressure}
Following chapter 4.4 in \cite{Girault:book}, we can construct a simple post processing which allows to approximate the pressure of the Stokes equations \eqref{eq:mixedstokes}. To this end we define the (discontinuous) pressure space
\begin{align*}
  Q_h^k := \mathcal{P}^k(\Omega, \rr) \cap L^2_0(\Omega,\rr),
\end{align*}
with the property $\divergence(V_h^{k-1}) = Q_h^{k-2}$. Further, there holds the (polynomial robust, see \cite{LedererSchoeberl2017}) Stokes inf-sup condition
\begin{align} \label{eq::infsup}
  \sup\limits_{0 \neq v_h \in V_h^{k-1}} \frac{\int_\Omega \divergence(v_h) q_h \dif x}{\| v_h\|_{1,h}} \gtrsim \|q_h\|_{L^2} \quad \forall q_h \in Q^{k-2}_h.
\end{align}
Now let $\sigma_h \in \Sigma_h^{k-1}$ be the solution of the HHJ-like method \eqref{eq::dischhjrotthree}, then we define the weak problem: Find $p_h \in Q^{k-2}_h$ such that
\begin{align} \label{eq::pressuredisc}
  \int_\Omega p_h \divergence(v_h) \dif x = - \int_{\Omega} f \cdot v_h \dif x + b(\sigma_h,v_h) \quad \forall v_h \in V_h^{k-1}.
\end{align}
\begin{theorem} \label{th::convergencep}
  Let $p \in L^2_0(\Omega, \rr) \cap H^{m-1}(\mesh, \rr)$ be the exact solution of \eqref{eq:mixedstokes}, and let $u,u_h,\sigma,\sigma_h$ be defined as in Theorem \ref{th::convergence}. Further, let $p_h$ be the solution of~\eqref{eq::pressuredisc}. For $s = \min(m-1, k-1)$ there holds the error estimate
  \begin{align*}
    \| p - p_h \|_{L^2} \lesssim h^s \| u \|_{H^{s+1}(\mesh)} + \| p \|_{H^s(\mesh)}.
  \end{align*}
\end{theorem}
\begin{proof}
The proof follows with exactly the same steps as the proof of Theorem 45 in chapter 4.4 in \cite{Girault:book} and involves \eqref{eq::infsup}, the results of Theorem \ref{th::convergence} and the properties of the interpolation operators of Lemma \ref{lem::intop}.
\end{proof}
%\section{A hybridization of the nt-component}
\section{Numerical example}
In this section we present a numerical example to validate the findings of Section \ref{sec::errorana}. All numerical examples were implemented within the finite element library Netgen/NGSolve, see \cite{netgen,ngsolve} and \url{www.ngsolve.org}. %Note, that the finite element spaces in Netge/NGSolve are constructed with respect to the discrete de Rham complex, thus an explicit basis for $W_h^{k,\perp}$ is available (see Remark~\ref{rem::nograds} and Remark \ref{rem::optconv}). This allows us to solve the reduced system \eqref{eq::dischhjrotthreenograd}.

Let $\Omega = (0,1)^d$ and $f = -\div(\sigma) + \nabla p$ with the exact solutions $u = \curl(\psi)$, $\sigma = \nu \nabla u$ and 
\begin{alignat*}{3}
 &\psi :=x^2(x-1)^2y^2(y-1)^2,& &\quad p := x^5+y^5 - \frac{1}{3},& &\quad \textrm{for } d = 2,\\
 &\psi := \begin{pmatrix}x^2(x-1)^2y^2(y-1)^2z^2(z-1)^2 \\x^2(x-1)^2y^2(y-1)^2z^2(z-1)^2\\x^2(x-1)^2y^2(y-1)^2z^2(z-1)^2 \end{pmatrix},&& \quad p := x^5+y^5+z^5 - \frac{1}{2},& &\quad \textrm{for } d = 3.
\end{alignat*}
In Table \ref{twodexample} and \ref{threedexample} we present several errors including their estimated order of convergence (eoc) for a fixed viscosity $\nu = 10^{-6}$, polynomial orders $k=2,3,4$ (where the order corresponds to the approximation space $W_h^k$ of the stream function, see Remark \ref{rem::optconv}) for the two and three dimensional case, respectively. As predicted by Theorem \ref{th::convergence} and Theorem \ref{th::convergencep} the $H^1$-seminorm error of the velocity $u_h$, the $L^2$-norm error of the stress $\sigma_h$ and the $L^2$-norm error of the pressure $p_h$ converge with optimal orders. Beside that, as given in the most right columns of Table \ref{twodexample} and \ref{threedexample}, we further observe that the $L^2$-norm error of the velocity $u_h$ converges at one order higher. This is explained by exploiting a standard Aubin-Nitsche duality argument to prove that the solution $u_h \in V^{k-1}_h$ fulfills the estimate
\begin{align*}
  \| u - u_h \|_{L^2} \lesssim h^k \|u\|_{H^k},
\end{align*}
whenever the problem admits full regularity and the exact solution is smooth enough.
\begin{table}[h]
\begin{center}
\footnotesize
%\begin{tabular}{r@{~}|@{~}c@{(}c@{)~}c@{(}c@{)}c@{(}c@{)~}c@{(}c@{)}}
\begin{tabular}{@{~}c@{~}|@{~~~}c@{~~~(}c@{)~~~}c@{~~~(}c@{)~~~}c@{~~~(}c@{)~~~}c@{~~~(}c@{)}}
      \toprule
  $|\mathcal{T}|$ & $|| \nabla u - \nabla u_h||_0$ & \footnotesize eoc &$|| \sigma - \sigma_h||_0$ & \footnotesize eoc &$|| p - p_h ||_0$ & \footnotesize eoc &$|| u - u_h||_0$ & \footnotesize eoc  \\
  \midrule
  \multicolumn{9}{c}{$k=2$}\\
  %\multicolumn{9}{c}{$u_h \in V^1_h, \sigma_h \in \Sigma_h^{1}, p_h \in Q_h^0$}\\~\\
56& \num{0.13001393187549476}&--& \num{0.11927413409736154}&--& \num{0.09143627315352611}&--& \num{0.007644034948413183}&--\\
224& \num{0.027812984862648382}&\numeoc{2.2248357393626854}& \num{0.015533712570080472}&\numeoc{2.940806631377618}& \num{0.05362568445886035}&\numeoc{0.7698424452479506}& \num{0.0005686071242203571}&\numeoc{3.7488302941835387}\\
896& \num{0.0076307013030320415}&\numeoc{1.8658710218285914}& \num{0.0030452600213262586}&\numeoc{2.3507653523653445}& \num{0.026983084003147504}&\numeoc{0.9908689077488075}& \num{6.505321464145376e-05}&\numeoc{3.127739919604422}\\
3584& \num{0.0038076618626950924}&\numeoc{1.002910288089736}& \num{0.001499657759541048}&\numeoc{1.0219321213373218}& \num{0.013513236488922484}&\numeoc{0.9976820001647605}& \num{1.6278375597645358e-05}&\numeoc{1.9986636085645344}\\
14336& \num{0.0019059463808071054}&\numeoc{0.9983978335172071}& \num{0.000747141087629903}&\numeoc{1.0051806898516025}& \num{0.006759343480347509}&\numeoc{0.9994182161163664}& \num{4.073243565777835e-06}&\numeoc{1.9987067515610044}\\
57344& \num{0.0009533954988531016}&\numeoc{0.9993608137614836}& \num{0.0003729354136632576}&\numeoc{1.002454901617547}& \num{0.003380012817933562}&\numeoc{0.9998544100131245}& \num{1.0182518251660338e-06}&\numeoc{2.0000836847559826}\\
\midrule
  \multicolumn{9}{c}{$k=3$}\\
  56& \num{0.00523034986584032}&--& \num{0.0012621316072134686}&--& \num{0.010726816952100888}&--& \num{9.164809454938274e-05}&--\\
224& \num{0.0019730198686088362}&\numeoc{1.40650196941348}& \num{0.00040498211530886767}&\numeoc{1.639932250728612}& \num{0.0032627015411204672}&\numeoc{1.7170831113967586}& \num{1.9911689551159115e-05}&\numeoc{2.2024892439454162}\\
896& \num{0.0004977857379270334}&\numeoc{1.9868086829640055}& \num{0.00010176062595352025}&\numeoc{1.9926787474469612}& \num{0.000822126046113274}&\numeoc{1.9886355154678856}& \num{2.502211748581717e-06}&\numeoc{2.992339854913078}\\
3584& \num{0.0001247947138501141}&\numeoc{1.9959680715004633}& \num{2.5665233547510327e-05}&\numeoc{1.987292156216387}& \num{0.00020594076383099926}&\numeoc{1.9971301762106557}& \num{3.1376413537413857e-07}&\numeoc{2.9954515216802426}\\
14336& \num{3.123062488269821e-05}&\numeoc{1.998523483572417}& \num{6.4533463951403885e-06}&\numeoc{1.9916960168279818}& \num{5.15108646183014e-05}&\numeoc{1.9992807634560525}& \num{3.9290651114319897e-08}&\numeoc{2.997422474248251}\\
57344& \num{7.810945070224002e-06}&\numeoc{1.9993924158870597}& \num{1.6185323749525798e-06}&\numeoc{1.9953612432748329}& \num{1.2879322249973986e-05}&\numeoc{1.9998200797907417}& \num{4.9160266489772245e-09}&\numeoc{2.998621434498366}\\
\midrule
  \multicolumn{9}{c}{$k=4$}\\
  56& \num{0.000681391672978328}&--& \num{0.0001302684843128565}&--& \num{0.0005701356016740986}&--& \num{5.408709018198976e-06}&--\\
224& \num{0.0001822254718591255}&\numeoc{1.9027596807508587}& \num{2.49550282082735e-05}&\numeoc{2.3840856570365836}& \num{0.00010986282366795334}&\numeoc{2.3756018150056817}& \num{9.173945306273816e-07}&\numeoc{2.5596700729644684}\\
896& \num{2.3474531364973576e-05}&\numeoc{2.956556374216378}& \num{3.154230905306352e-06}&\numeoc{2.9839703530903487}& \num{1.3825268848174201e-05}&\numeoc{2.9903238365461293}& \num{5.974378828219136e-08}&\numeoc{3.940681681078165}\\
3584& \num{2.9607670530971787e-06}&\numeoc{2.987053464956877}& \num{3.9463370988444426e-07}&\numeoc{2.9987021739622546}& \num{1.7310660672100465e-06}&\numeoc{2.9975748433699128}& \num{3.779613591687419e-09}&\numeoc{3.9824780667621154}\\
14336& \num{3.7093528837768137e-07}&\numeoc{2.996731558647405}& \num{4.9261491018216656e-08}&\numeoc{3.0019819970669888}& \num{2.1647426730970688e-07}&\numeoc{2.9993933426356945}& \num{2.3704732630020193e-10}&\numeoc{3.994991722392563}\\
57344& \num{4.6392967413546165e-08}&\numeoc{2.999189490281595}& \num{6.1537411035973574e-09}&\numeoc{3.000924641269482}& \num{2.7062128786387647e-08}&\numeoc{2.99984830384561}& \num{2.459578346269868e-11}&\numeoc{3.268692204891301}\\
\bottomrule
\end{tabular}
\caption{The $H^1$-seminorm and the $L^2$-norm error of the discrete velocity $u_h \in V^{k-1}_h$, and the $L^2$-norm error of the discrete pressure $p_h\in Q^{k-2}_h$ and the discrete stress $\sigma_h\in \Sigma^{k-1}_h$ for a fixed viscosity $\nu = 10^{-6}$ and different polynomial order $k=2,3,4$ in the two dimensional case.} \label{twodexample}
\end{center}
\end{table}

\begin{table}[h]
\begin{center}
\footnotesize
%\begin{tabular}{r@{~}|@{~}c@{(}c@{)~}c@{(}c@{)}c@{(}c@{)~}c@{(}c@{)}}
\begin{tabular}{@{~}c@{~}|@{~~~}c@{~~~(}c@{)~~~}c@{~~~(}c@{)~~~}c@{~~~(}c@{)~~~}c@{~~~(}c@{)}}
      \toprule
  $|\mathcal{T}|$ & $|| \nabla u - \nabla u_h||_0$ & \footnotesize eoc &$|| \sigma - \sigma_h||_0$ & \footnotesize eoc &$|| p - p_h ||_0$ & \footnotesize eoc &$|| u - u_h||_0$ & \footnotesize eoc  \\
  \midrule
  \multicolumn{9}{c}{$k=2$}\\
  %\multicolumn{9}{c}{$u_h \in V^1_h, \sigma_h \in \Sigma_h^{1}, p_h \in Q_h^0$}\\~\\
6& \num{0.004434122873853067}&--& \num{0.004470307143166375}&--& \num{0.4249981025833159}&--& \num{0.0005456458429224737}&--\\
48& \num{0.004527537661506566}&\numeoc{-0.030077892783735804}& \num{0.003591467999820266}&\numeoc{0.31580029754873734}& \num{0.2942601355870876}&\numeoc{0.5303642915561634}& \num{0.0004032761324498577}&\numeoc{0.43619683515708724}\\
384& \num{0.0024306963364550877}&\numeoc{0.8973569713043908}& \num{0.0016259309279495345}&\numeoc{1.1433076902788248}& \num{0.16486930866818672}&\numeoc{0.8357692506287094}& \num{0.0001076861457648633}&\numeoc{1.904935370030929}\\
3072& \num{0.0012578683759000234}&\numeoc{0.9503887050775189}& \num{0.0008067808576680761}&\numeoc{1.0110172118810072}& \num{0.085007541206985}&\numeoc{0.9556601214168878}& \num{2.82649600576796e-05}&\numeoc{1.9297460901989252}\\
24576& \num{0.0006357499952301691}&\numeoc{0.9844495144383602}& \num{0.00040290206429712666}&\numeoc{1.0017476565581578}& \num{0.042836930059562885}&\numeoc{0.9887357396151178}& \num{7.175755570123929e-06}&\numeoc{1.977812005005429}\\
196608& \num{0.00031902614752816176}&\numeoc{0.9947848733134577}& \num{0.00020143917727951153}&\numeoc{1.0000849013472959}& \num{0.02146047334372828}&\numeoc{0.9971731946780493}& \num{1.8019408621545757e-06}&\numeoc{1.9935790839866654}\\
\midrule
  \multicolumn{9}{c}{$k=3$}\\
6& \num{0.0061032273801362655}&--& \num{0.004450364689785478}&--& \num{0.18870050314187617}&--& \num{0.0005712178697702683}&--\\
48& \num{0.0021353001932910088}&\numeoc{1.5151334332421376}& \num{0.0009757412337626561}&\numeoc{2.189353062676604}& \num{0.07499527783970168}&\numeoc{1.3312266069774374}& \num{7.410004704159338e-05}&\numeoc{2.946494749292355}\\
384& \num{0.0006446331195674506}&\numeoc{1.7278886894936918}& \num{0.00030126110824357446}&\numeoc{1.6954841582224178}& \num{0.021097208327425423}&\numeoc{1.8297476495837746}& \num{1.23277558255912e-05}&\numeoc{2.587564266452913}\\
3072& \num{0.00017529955918102472}&\numeoc{1.8786559439949944}& \num{8.043019689365818e-05}&\numeoc{1.9052052811787847}& \num{0.005428310534654722}&\numeoc{1.9584769482355937}& \num{1.6723278096966175e-06}&\numeoc{2.8819806146240787}\\
24576& \num{4.493624261026761e-05}&\numeoc{1.9638709659558737}& \num{2.056157726746341e-05}&\numeoc{1.9677863146406278}& \num{0.0013668098199461753}&\numeoc{1.989690736587335}& \num{2.1486966658698875e-07}&\numeoc{2.960323936987048}\\
196608& \num{1.1316236281040374e-05}&\numeoc{1.9894852917622727}& \num{5.1843266204728604e-06}&\numeoc{1.987722418949129}& \num{0.0003423123604024178}&\numeoc{1.9974272260416395}& \num{2.7139259971408313e-08}&\numeoc{2.985008542871208}\\
\midrule
  \multicolumn{9}{c}{$k=4$}\\
6& \num{0.002603430400014771}&--& \num{0.001002733452237441}&--& \num{0.03218835776187978}&--& \num{0.00011590700833582803}&--\\
48& \num{0.000626645511328069}&\numeoc{2.05469238119103}& \num{0.00019383013381372408}&\numeoc{2.3710733760771325}& \num{0.0054374914083679575}&\numeoc{2.5655258515252144}& \num{1.3204850671736683e-05}&\numeoc{3.1338279099076605}\\
384& \num{0.000125781463241854}&\numeoc{2.316730228029362}& \num{2.9589233521784386e-05}&\numeoc{2.7116486462439244}& \num{0.0007267803073645454}&\numeoc{2.903349981933511}& \num{1.2550487969830667e-06}&\numeoc{3.3952526235974974}\\
3072& \num{1.7497485173811685e-05}&\numeoc{2.8456998332556687}& \num{3.941516334036913e-06}&\numeoc{2.9082496671725084}& \num{9.233451338386949e-05}&\numeoc{2.9765774156860187}& \num{8.766348762400829e-08}&\numeoc{3.839623571302862}\\
24576& \num{2.2491894941721614e-06}&\numeoc{2.959670467222739}& \num{5.054872885635757e-07}&\numeoc{2.9630040348267537}& \num{1.158839158212131e-05}&\numeoc{2.994189667941557}& \num{5.642542588408196e-09}&\numeoc{3.957558768861314}\\
196608& \num{2.833055405007371e-07}&\numeoc{2.9889744919128782}& \num{6.396300205613679e-08}&\numeoc{2.9823652559384985}& \num{1.4500053288137864e-06}&\numeoc{2.9985502331901914}& \num{3.5673499768399066e-10}&\numeoc{3.983420738281588}\\
\bottomrule
\end{tabular}
\caption{
  The $H^1$-seminorm and the $L^2$-norm error of the discrete velocity $u_h \in V^{k-1}_h$, and the $L^2$-norm error of the discrete pressure $p_h\in Q^{k-2}_h$ and the discrete stress $\sigma_h\in \Sigma^{k-1}_h$ for a fixed viscosity $\nu = 10^{-6}$ and different polynomial order $k=2,3,4$ in the three dimensional case.}\label{threedexample}
  %The $H^1$-seminorm error of the discrete velocity $u_h \in V^{k-1}_h$, the $L^2$-norm error of the discrete pressure $p_h\in Q^{k-2}_h$ and the discrete stress $\sigma_h\in \Sigma^{k-1}_h$ and the $L^2$-norm error of the discrete velocity for a fixed viscosity $\nu = 10^{-6}$ and different polynomial order $k=2,3,4$ in the three dimensional case.} 
\end{center}
\end{table}

%\section{Conclusions}
%\label{sec:conclusions}

\bibliographystyle{siamplain}
\bibliography{references}
%------------------------------------------------------------------------------
\end{document}